\documentclass[11pt, a4paper]{article}
\usepackage[utf8]{inputenc}
\usepackage[english]{babel}
\usepackage{amsfonts}
\usepackage{mathtext}
\usepackage{amsthm}
\usepackage{amsmath}
\usepackage{amssymb}
\usepackage[left=2cm,right=2cm,top=2cm,bottom=2cm,bindingoffset=0cm]{geometry}
\usepackage{indentfirst}
\usepackage{color}
\usepackage{xcolor}
\usepackage{bm}
\usepackage[shortcuts]{extdash} 

\newtheorem{theorem}{Theorem}
\newtheorem{lemma}{Lemma}
\newtheorem{corollary}{Corollary}
\newtheorem{proposition}{Proposition}
\newtheorem{remark}{Remark}
\newtheorem{example}{Example}

\begin{document}
	
	\title{The probability of reaching a receding boundary by \\  branching random walk with fading branching and\\  heavy-tailed jump distributions\footnote{The work is supported by the RScF grant 17-11-01173-extension }}
	
	\author{P.I. Tesemnivkov\footnote{MCA, Novosibirsk State University and S.L. Sobolev Institute of Mathematics, Email: tesemnikov.p@gmail.com} \and S.G. Foss\footnote{Heriot-Watt University, Novosibirsk State University  and S.L. Sobolev Institute of Mathematics, Email: sergueiorfoss25@gmail.com}
	}
	
	\maketitle
	\begin{abstract}
	Foss and Zachary (2003) and Foss, Palmowski and Zachary (2005) studied the probability of achieving a receding boundary on a time interval of random length by a random walk with a heavy-tailed jump distribution. They have proposed and developed a new approach that allows to generalise  results of Asmussen (1998) onto the case
of arbitrary stopping times and a wide class of nonlinear boundaries, and to obtain uniform results  over all stopping times.

In this paper, we consider a class of branching random walks with fading branching and obtain results on the tail asymptotics for the maximum of a branching random walk on a time interval of random (possibly unlimited) length, as well as uniform results within a class of bounded random time intervals.
		
		

		\vskip 0.2cm
		\noindent
		{\it Short title:} Branching random walk with heavy tails
		
	\end{abstract}
	
	\section{Introduction} \label{Intro}
	
	Let $ \xi,\xi_0,\xi_1,\xi_2,\ldots$ be a sequence of independent random variables (r.v.'s)  with the   common distribution function $ F $ and zero mean: 
	\begin{align*}
		\mathbb{E} \xi = 0.
	\end{align*}
	 We assume  that  $ F $ has a heavy (right) tail ($ F \in \mathcal{H} $), i.e. 
	\begin{align*}
		\mathbb{E} e^{\lambda \xi} \equiv \int_{-\infty}^{\infty} e^{\lambda t} F(dt) = \infty
	\ \ \mbox{for all} \ \   \lambda > 0 .
	\end{align*}
	Consider a random walk
	\begin{align*}
		S_{0} = 0, \qquad S_{n} = \sum_{k=1}^{n} \xi_{k}, \quad n \ge 1
	\end{align*}
	and a non-negative function 
	 $ g $ that is defined on the set of non-negative integers  $ \mathbb{Z}_{+} $. We call a sequence 
	\begin{align} \label{Intro::shifted_RW}
		S^{g}_{n} = S_{n} - g(n), \quad n \ge 0
	\end{align}
 a ``$g$-shifted'' random walk.  Let 
	\begin{align*}
		M_{n}^{g} = \max_{0 \le k \le g} S_{k}^{g}, \quad n \ge 0
	\end{align*}
	be a sequence of its partial maxima. 
	
	For any non-negative integer-valued r.v. 
	 $ \mu \le \infty $ denote by 
	\begin{align*}
		M_{\mu}^{g} = \max_{0 \le k \le \mu} S_{k}^{g}
	\end{align*}
	the maximum of its partial sums within the time interval  $[0, \mu]$, and introduce a function 
	\begin{align} \label{Intro::old_H}
		H_{\mu}^{g} (x) = \sum_{n=1}^{\infty} \mathbb{P} \left( \mu \ge n \right) \overline{F} (x + g(n)),
	\end{align}
	where $ \overline{F} (x):=1-F(x) $ is the (right) tail of the distribution $ F $.
	
	Foss, Palmowski and Zachary  \cite{FossPalmZach2005} studied conditions for either the lower bound 
	\begin{align} \label{Intro::old_goal1}
		\mathbb{P} \left( M^{g}_{\mu} > x \right) \ge (1+o(1)) H_{\mu}^{g} (x)
	\end{align}
	or the asymptotic equivalence
	\begin{align} \label{Intro::old_goal2}
		\mathbb{P} \left( M^{g}_{\mu} > x \right) = (1+o(1)) H_{\mu}^{g} (x)
	\end{align}
	to hold uniformly within certain sufficiently broad  
	classes of  $ \mu $ and $ g $.
	
	Recall the definitions and some properties of three main classes of heavy-tailed distributions. Distribution
	 $ G $ on $ \mathbb{R} $ has a {\it long  (right) tail} ($ G \in \mathcal{L} $) if $ \overline{G} (x) > 0 $ for all $ x > 0 $ and 
	$\overline{G}(x+h) 
	\sim \overline{G}(x)$,\footnote{ 
 For any two positive functions $f_1(x)$ $f_2(x)$, we write $f_1(x)\sim f_2(x)$ if 	
	$
		\frac{f_1(x)}{f_2(x)} \to 1 \text{ as } x \to \infty.
	$\\
	It follows from $\overline{G}(x+h)\sim \overline{G}(x)$ for {\it some} $h>0$ and from the monotonicity of function $\overline{G}(x)$ that this equivalence holds also  for {\it all} $h>0$ and, moreover, that there exists a positive function $h(x)\uparrow\infty$ such that $\overline{G}(x+h(x))\sim \overline{G}(x)$ (see  \cite[Lemma 2.19]{FossKorZach2013}). In this case we say that the distribution function  $G$ is $h$-{\it insensitive}. A detailed analysis of properties of $h$-insensitivity may be found, e.g., in  \cite{FossRich2010}.}
	for any fixed $ h > 0 $. 
	If $ G $ is long-tailed and if the mean of its restriction to the positive half-line is finite, 
	\begin{align*}
		m_{G^{+}} := \int_{0}^{\infty} \overline{G}(y)dy < \infty,
	\end{align*}
	then (see, e.g., 
	\cite[Lemma 2.26]{FossKorZach2013}) {\it the distribution of the integrated tail }
	\begin{align*}
		\overline{G}_{I} (x) = \min \{ 1, \int_{x}^{\infty} \overline{G}(y) dy \}
	\end{align*}
	is also long-tailed  and 
	 $ \overline{G} (x) = o \left( \overline{G}_{I} (x) \right) $.
	
	Distribution $ G $ on $ \mathbb{R} $ is {\it subexponential} ($ G \in \mathcal{S} $) if $ G \in \mathcal{L} $ and
	\begin{align*}
		\overline{G * G} (x) \sim 2 \overline{G} (x).
	\end{align*}
	Notice that if $ G \in \mathcal{S} $, then $ \overline{G^{*n}}(x) \sim n \overline{G} (x) $, for all $ n \ge 1 $.
	
	Distribution $ G $ on $ \mathbb{R} $ is {\it strong subexponential} ($ G \in \mathcal{S}^{*} $) if $ \overline{G}(x) > 0 $ for all $ x >0 $, $ m_{G^{+}} $ is finite and
	\begin{align*}
		\int_{0}^{x} \overline{G}(x-y) \overline{G} (y) dy \sim 2 m_{G^{+}} \overline{G}(x).
	\end{align*}
	
	It is known (see, e.g. \cite{Klup1988} or \cite[Theorem 3.27]{FossKorZach2013}) that if $ G \in \mathcal{S}^{*} $, then both $ G $ and $ G_{I} $ belong to the class $ \mathcal{S} $. It is also known (see  \cite[Lemma 2.23, Theorem 3.11 and Theorem 3.25]{FossKorZach2013}), that each of the classes of distributions  ($ \mathcal{L}, \mathcal{S} $ and $ \mathcal{S}^{*} $) 
	has the following {\it closure} property:
	if a distribution $ F $ belongs to any of these classes and if  $ \overline{F}(x) \sim \overline{G} (x) $ for another distribution $G$, then $ G $ belongs to the same class.
	
	Let $ \mathcal{F} $ be the family of all non-negative integer-valued r.v.'s. $\sigma$ that do not depend on the future\footnote{By the {\it independence of the future},  $\sigma\in {\cal F}$, we mean the following: for any  $n=1,2,\ldots$ the event $\{\sigma >n\}$ does not depend on the family of r.v.'s  $\{\xi_k\}_{k>n}$.} w.r. to the natural filtration of the sigma-algebras generated by the sequence $ \{ \xi_{n} \}_{n \ge 0} $.
	
	For any $ \sigma \in \mathcal{F} $, let 
	\begin{align*}
		\mathcal{F}_{\sigma} = \{ \mu \in \mathcal{F}: \mu \le \sigma \ \mbox{a.s.} \}.
	\end{align*}
	In particular, for any $ N > 1 $, the family $ \mathcal{F}_{N} $ consists of independent of the future r.v.'s $ \mu $ that are a.s. bounded above by the constant $ N $.
	
	For any constant $ c \in \mathbb{R} $, let  $ \mathcal{G}_{c} $ be the family of all non-negative functions $ g $ such that 
	\begin{align*}
		g(1) \ge c \qquad \text{and} \qquad g(n+1) \ge g(n) + c \ \ \ \mbox{for all} \ \ n \ge 1.
	\end{align*}

	We recall now the main results from  \cite{FossPalmZach2005}.
	\begin{proposition}[Theorem 1 in  \cite{FossPalmZach2005}] \label{Intro::old_th1}~
		\begin{enumerate}
			\item[(i)] Assume that $ F \in \mathcal{L} $. Then, for any integer $ N > 1 $, the inequality  \eqref{Intro::old_goal1} holds uniformly in all $ \mu \in \mathcal{F}_{N} $ and $ g \in \mathcal{G}_{0} $.
			\item[(ii)] Assume that  $ F \in \mathcal{S} $. Then, for any integer $ N > 1 $, the equivalence   \eqref{Intro::old_goal2} holds uniformly in all $ \mu \in \mathcal{F}_{N} $ and $ g \in \mathcal{G}_{0} $.
		\end{enumerate}
	\end{proposition}

	\begin{proposition}[Theorem 2 in \cite{FossPalmZach2005}] \label{Intro::old_th2}~
		\begin{enumerate}
			\item[(i)] Assume that $ F \in \mathcal{L} $. Then, for any $ c > 0 $, the inequality  \eqref{Intro::old_goal1} holds uniformly in all   $ \mu \in \mathcal{F} $ and $ g \in \mathcal{G}_{c} $.
			\item[(ii)] Assume that $ F \in \mathcal{S}^{*}$. Then, for any $ c > 0 $, the equivalence   \eqref{Intro::old_goal2} holds uniformly  in all  $ \mu \in \mathcal{F} $ and $ g \in \mathcal{G}_{c} $.
		\end{enumerate}
	\end{proposition}

	In particular, Propositions \ref{Intro::old_th1} and \ref{Intro::old_th2} generalise a number of earlier  known results on the  asymptotics of the tail distribution $ \mathbb{P} \left( M_{\mu}^{g} > x \right) $. 
	
	Let $ g(n) = \widehat{c} (n) := cn $ where $c>0$ is an arbitrary constant. Consider a random walk 
	\begin{align*}
		S_{n}^{\widehat{c}} = \sum_{i=1}^{n} (\xi_{i} - c)
	\end{align*}
	with negative drift $ -c $.
	Proposition \ref{Intro::old_th2} implies that if  $ F \in \mathcal{S}^{*} $ and $ \mathbb{E} \mu < \infty $, then, for any $ c > 0 $,
	\begin{align}  \label{Intro::FZ_res}
		\mathbb{P} \left( M_{\mu}^{\widehat{c}} > x \right) \sim H_{\mu}^{\widehat{c}} (x) \sim \mathbb{E} \mu \overline{F} (x),
	\end{align}
	where $ H_{\mu}^{\widehat{c}} (x) = \sum_{n \ge 1} \mathbb{P} \left( \mu \ge n \right) \overline{F} (x + cn) $.
	The equivalences \eqref{Intro::FZ_res} were obtained by Asmussen \cite{Asm1998} in the case where  $ \mu = \tau = \inf \{ n \ge 1: S_{n}^{\widehat{c}} < 0 \} $ and then extended onto arbitrary stopping times $ \mu $ in \cite{FossZach2003}. It was also shown in \cite{FossZach2003} that the condition  $ F \in \mathcal{S}^{*} $ is necessary for the equivalence  \eqref{Intro::FZ_res} to hold in  the case $ \mu = \tau $.
	
	On the other hand, if $ F \in \mathcal{S}^{*} $ and  $ \mu = \infty $ a.s., then Proposition  \ref{Intro::old_th2} implies that, for any $ c > 0 $,
	\begin{align} \label{Intro::Veraver_res}
		\mathbb{P} \left( \sup_{n \ge 0} S_{n}^{\widehat{c}} > x \right) \sim H_{\infty}^{\widehat{c}} (x) \sim \frac{1}{c} \overline{F}_{I} (x).
	\end{align}
	Here the equivalence  \eqref{Intro::Veraver_res} is a particular case of the result by Veraverbeke  \cite{Verav1977}. It is known that  \eqref{Intro::Veraver_res} holds if and only if $ \overline{F}_{I} \in \mathcal{S} $ (see  \cite{Kor2002}). Natural analogues of the results from \cite{FossPalmZach2005} in continuous time have been obtained in  \cite{FoKoPaRo2017}.
	
	We turn now to a brief description of a {\it branching random walk} (BRW).  Let $ Z_{n} $ be a branching process in changing environment (this means that the distribution of the number of offspring in this process may depend on a generation), and let $ \Pi $ be the family of all finite paths in its genealogical tree that start from its root.  Then the corresponding branching random walk $ \{S(\pi)\}_{\pi\in\Pi} $ 
	with independent increments having zero-mean distribution $ F $ is defined by 
	\begin{align*}
		S(\pi) = \sum_{e \in \pi} \xi_{e},
	\end{align*}
	where r.v.'s $  \xi_{i,k} $ are independent and have the common  distribution $ F $, and $ \xi_{e} $ is a r.v. $ \xi_{k(e), i(e)} $, where $ k(e) $ is the number of the generation the edge $e$ ends in and  $ i(e) $ is the number of the vertex in generation $ k(e) $  the edge $ e $ ends at.  We also assume  that all r.v.'s $ \xi_{i, k} $ do not depend on $ Z_{n} $.
	
	For a non-negative function $ g $ on $\mathbb{Z}_{+}$, let
	\begin{align*}
		S^{g}(\pi) = S(\pi) - g (|\pi|)
	\end{align*}
	where $ |\pi| $ is the length of (i.e. the number of the edges in) the path $ \pi $.  Let 
	\begin{align*}
		R_{n}^{g} = \max_{\pi: |\pi| \le n} S^{g} (\pi)
	\end{align*}
	be the rightmost point of the  {\it $ g $-shifted BRW} in the first $ n $ generations.
	
	For any $ x $ and for any non-negative integer-valued r.v. $ \mu $, let 
	\begin{align} \label{Intro::new_H}
		H_{\mu}^{g} (x; \widehat{\mathcal{P}}) = \sum_{n=1}^{\infty} \mathbb{E} \left[ Z_{n} \mathbb{I} (\mu \ge n) \right] \overline{F}(x + g(n)),
	\end{align}
	where $ \widehat{\mathcal{P}} = \left( \mathcal{P}_{0}, \mathcal{P}_{1}, \mathcal{P}_{2}, \ldots \right) $ and  $ \mathcal{P}_{k} $ is the distribution of the number of offspring in the $k$'th generation of the branching process  $ Z_{n} $. Notice that if there is no branching, then the functions  \eqref{Intro::old_H} and \eqref{Intro::new_H} coincide:
	\begin{align*}
		H_{\mu}^{g} (x; \widehat{\delta}) = H_{\mu}^{g}(x),
	\end{align*}
	where $ \widehat{\delta} = (\delta_{1}, \delta_{1}, \delta_{1}, \ldots) $ is the sequence of degenerative  at point 1 distributions.
	
	In this paper, we study conditions that imply the ``individual'' asymptotic equivalence  
	\begin{align} \label{Intro::goal_non_un2}
		\mathbb{P} \left( R_{\mu}^{g} > x \right) \sim H_{\mu}^{g} (x; \widehat{\mathcal{P}})
	\end{align}
	for given $\mu$ and $g$, and also uniform generalisations of Proposition 1 
	\begin{align} \label{Intro::goal_un1}
		\mathbb{P} \left( R_{\mu}^{g} > x \right) \ge (1 + o(1)) H_{\mu}^{g} (x; \widehat{\mathcal{P}})
	\end{align}
	and
	\begin{align} \label{Intro::goal_un2}
		\mathbb{P} \left( R_{\mu}^{g} > x \right) = (1 + o(1)) H_{\mu}^{g} (x; \widehat{\mathcal{P}})
	\end{align}
	onto certain classes of 
	$ \mu $ and $ g $. Note that if $ \widehat{\mathcal{P}} = \widehat{\delta} $, then $ R_{\mu}^{g} $ conicides with $ M_{\mu}^{g} $ from \cite{FossPalmZach2005} and, therefore,  \eqref{Intro::goal_non_un2} is equivalent (depending on the type of $ \mu $) either to  \eqref{Intro::FZ_res} or to \eqref{Intro::Veraver_res}, while \eqref{Intro::goal_un1} and \eqref{Intro::goal_un2} are equivalent to \eqref{Intro::old_goal1} and \eqref{Intro::old_goal2}, correspondingly.

	The papers \cite{Durr1983} and \cite{Gantert2000} considered supercritical time-homogeneous BRW's with heavy-tailed jump distributions (that were regularly varying in  \cite{Durr1983} and semi-exponential in  \cite{Gantert2000}) and studied the asymptotic behaviour of the rightmost point in the 
	$n$'th generation, as $n$ grows. In both cases it was shown that the growth rate is superlinear and the 
	corresponding limit theorems for normalised sequences have been proven.
	
	In Appendix A2 below, we show that the the supremum of the rightmost points in all generations of a supercritical BRW is infinite not only in this cases, but for any heavy-tailed jump distribution. 
	For the supremum to be finite, we introduce a condition of branching fading (see condition 
		 \eqref{BPVE::fad_cond} below).

Our paper includes four Sections and Appendix. In Section \ref{MMR} we present a detailed description of a branching process in varying environment and of the corresponding branching random walk, and formulate our main results. In Section \ref{Examples}, we provide corollaries of our main results and two examples. Section 4 includes the proofs of the main results. Appendix contains the proofs of auxiliary results  \ref{BPVE::nu_moments} and  \ref{BPVE::Z_moments} (that are formulated below) and also the proof of the fact that the supremum of a supercritical BRW with heavy-tailed jump distributions is necessarily infinite. 
	
	\section{Model Description and Main Results} \label{MMR}
	
	In this Section we present a complete formal description of the model and formulate our main results. 
	
	\subsection{Branching processes in varying environment} \label{BPVE}
	
 {\it A branching process in varying environment} $ \{ Z_{n} \}_{n \ge 0} $ (in other words, a time-inhomogeneous Galton-Watson process) is a random process in discrete time defined by recursion 
	\begin{align*}
		Z_{0} = 1, \qquad Z_{n+1} = \sum_{j=1}^{Z_{n}} \zeta_{n, j}, \ n \ge 0,
	\end{align*}
	where, for any  $ n \ge 0 $, sequence 
	$ \widehat{\zeta}_{n} := \{ \zeta_{n,j} \}_{j \ge 0} $ is a sequence of independent and identically distributed (i.i.d.) r.v.'s with common distribution $ \mathcal{P}_{n} $ and the 
	sequences $ \widehat{\zeta}_{1}, \widehat{\zeta}_{2}, \ldots $ are mutually independent.
	
	We assume that the process $ \{ Z_{n} \}_{n \ge 1} $ dies out with probability $ 0 $:
	\begin{align} \label{BPVE::non_ext}
		\zeta_{n,1} \ge 1 \text{ a.s. for all } n \ge 0.
	\end{align}
	We assume also that
	\begin{align} \label{BPVE::finite_mean}
		\mathbb{E} \zeta_{n,1} < \infty \qquad \text{ for all } n \ge 0.
	\end{align}

We restrict our consideration by the class of so-called {\it processes with fast fading}, namely, 
the processes satisfying the following condition: 
	\begin{align} \label{BPVE::fad_cond}
		L:= \prod_{n=0}^{\infty} \mathbb{E} \zeta_{n,1} < \infty.
	\end{align}
	
	By the Beppo Levy theorem, conditions 
	 \eqref{BPVE::non_ext} and \eqref{BPVE::fad_cond} imply the existence of an integrated r.v. 
	  $ Z $ such that $ Z_{n} \to Z $ a.s. and in $ L_{1} (\Omega) $, and, in particular,
	 $ \mathbb{E} Z_{n} \to L $ as $ n \to \infty $. Moreover, under the condition 
	\eqref{BPVE::fad_cond} {\it the fading time}
	\begin{align*}
		\nu := \inf \{ n \ge 1: Z_n = Z_{n+1} = Z_{n+2} = \ldots = Z \}
	\end{align*}
	is finite a.s.
	
	Note that, under the condition \eqref{BPVE::finite_mean}, a sequence $ Z_{n} / L_{n} $, where $ L_{n} = \prod_{k=0}^{n} \mathbb{E} \zeta_{k,1} $, forms a non-negative martingale and, therefore,
	converges with probability 1 to some (possibly, infinite) limit   $ Z $.
	
	Let
	\begin{align*}
		q_{n} = \mathbb{P} (\zeta_{n,1} \neq 1).
	\end{align*}
	It is easy to see that condition \eqref{BPVE::fad_cond} implies convergence 
	\begin{align} \label{BPVE::q_n_conv}
	\sum_{n=0}^{\infty} q_{n} < \infty
	\end{align}
	and, therefore, convergence $ q_{n} \to 0 $ as $ n \to \infty $.
	
	Note that it was shown in \cite[
	Theorem 1.4]{KersVat2017} that if a branching process $ Z_{n} $ satisfies condition  \eqref{BPVE::finite_mean} only, then condition \eqref{BPVE::q_n_conv} is equivalent to 
	\begin{align*}
		\mathbb{P} \left( Z = 0 \right) + \mathbb{P} \left( Z = \infty \right) < 1.
	\end{align*}
	
	We provide now necessary and sufficient conditions for finiteness of power and exponential moments of the fading time  $ \nu $. Clearly, the rate of convergence of $ q_{n} $ to $ 0 $ determines its existence.
	\begin{proposition} \label{BPVE::nu_moments}
	Assume that conditions
		\eqref{BPVE::non_ext} and \eqref{BPVE::fad_cond} hold,  and let $ f: \mathbb{R}^{+} \to \mathbb{R}^{+} $ be an arbitrary non-decreasing function. Then
		\begin{align*}
		\mathbb{E} f(\nu) < \infty \text{ if and only if } \sum_{n=0}^{\infty} f(n+1) q_n < \infty.
		\end{align*}
		In particular, for any $ s, \lambda > 0 $,
		\begin{enumerate}
			\item $ \mathbb{E}  \nu^{s} < \infty $ if and only if $ \sum_{n=0}^{\infty} n^s q_n < \infty $;
			\item $ \mathbb{E} e^{\lambda \nu} < \infty $ if and only if $ \sum_{n=0}^{\infty} e^{\lambda n} q_n < \infty $.
		\end{enumerate}
	\end{proposition}
	
	Note that \eqref{BPVE::fad_cond} provides only a sufficient condition for the existence of power moments of r.v. $ \nu $, as the following example shows. Assume that, for any $ n \ge 3 $, r.v. 
	$ \zeta_{n,1} $ takes two values only, value $ 2 $ with probability $ q_{n} = (n \ln^2 n)^{-1} $ and value 
	$ 1 $ with probability $ 1 - q_{n} $. Then \eqref{BPVE::fad_cond} holds, however 
	$ \mathbb{E} \nu^{s} = \infty $ for any $ s > 0 $, as follows from \eqref{BPVE::nu_moments}.
	
	\begin{proposition} \label{BPVE::Z_moments}
		Assume that conditions \eqref{BPVE::non_ext} and \eqref{BPVE::fad_cond} hold and that 
		\begin{align*}
		\prod_{n=0}^{\infty} \mathbb{E} \zeta_{n,1}^{s} < \infty, 
		\end{align*}
		for some $ s > 1 $.
		Then $ \mathbb{E}Z^{s} < \infty $.
	\end{proposition}
	
	Note that Propositions \ref{BPVE::nu_moments} and \ref{BPVE::Z_moments} are close to known results of Lindvall \cite{Lind1974} and Kersting \cite{Kers2017}. For the sake of completeness, we present their proofs in Appendix. We apply Propositions \ref{BPVE::nu_moments} and \ref{BPVE::Z_moments} in the examples below.
	
	\subsection{Branching random walks in changing environment}
	
	Now we describe in more detail a BRW on a branching process  $ Z_{n} $.  Consider a sequence  $ \{ \xi, \xi_{i,k} \}_{i, k\ge 0} $ of i.i.d. r.v.'s with common distribution function  $ F $ having zero mean:
	\begin{align} \label{BRW::zero_mean}
		\mathbb{E} \xi = 0.
	\end{align}
	 We will assume that 
	 \begin{align} \label{BRW::indep}
	 	 \{ \xi_{i,k} \} \text{ does not depend on } \{ Z_{n} \}_{n \ge 0}.
	 \end{align}
	 
	Further, let $ \mathcal{T} $ be a genealogical tree of the process $ Z_{n} $ and $ \pi = (e_{1}, \ldots, e_{n}) $ an arbitrary path in the tree $ \mathcal{T} $ that starts from its root. Define the process $ S(\pi) $ as follows:
	\begin{align*}
		S(\emptyset) = 0, \qquad \qquad S(\pi) = \sum_{k=1}^{n} \xi_{i(e_{k}), k} \ \ \mbox{for} \ \  n= |\pi|  \ge 1,
	\end{align*}
	where, for any $ k=1, \ldots, n $, $ i(e_{k}) $ is the number of the vertex in generation $ k $
	where the edge $ e_{k} $ ends. In what follows, we use (depending on the context) both notation,  $ \xi_{i(e), k} $ or $ \xi_{e} $, for the increment of the BRW that correspond to the edge  $ e $ that ends in the generation $ k $.
	
	Analogous  with \eqref{Intro::shifted_RW}, for any non-negative function  $ g $ on the set of non-negative integers, we define the $ g $-shifted BRW as 
	\begin{align*}
		S^{g}(\pi) = S(\pi) - g(|\pi|).
	\end{align*}
	Let $ \mu \le \infty $ be an integer-valued non-negative r.v. We will study the asymptotic behaviour of the tail distribution of the r.v.   
	\begin{align*}
		 R_{\mu}^{g} = \max_{\pi: |\pi| \le \mu} S^{g}(\pi),
	\end{align*}

	We consider two cases: 
	\begin{enumerate}
		\item[(HM)] \label{MR::ind_crv} Random time $ \mu $ does not depend on the sequence of the increments $ \{ \xi_{n,i} \} $ of the BRW (however it may depend on the process $ Z_{n} $);
		\item[(MO)] \label{MR::st_crv} Random time $ \mu $ does not depend on the future of the sequence of the increments  $ \{ \xi_{n,i} \} $, i.e. for any $ n \ge 0 $ and for any events $ A \in \sigma(\{ \xi_{k,i} \}_{k \le n}, \mathbb{I} (\mu \le n), \mathcal{T}) $ and $ B \in \sigma(\{\xi_{k,i}\}_{k>n}, \mathcal{T}) $,
		\begin{align*}
			\mathbb{P} \left( AB | \mathcal{T} \right) = \mathbb{P} \left( A | \mathcal{T} \right) \mathbb{P} \left( B | \mathcal{T} \right) \quad \mbox{a.s.}
		\end{align*}
	\end{enumerate}

	Denote by $ \mathcal{F} (\mathcal{T}) $ the class of all non-negative integer-valued r.v.'s satisfying condition (MO) and let 
	\begin{align*}
		\mathcal{F}_{N} (\mathcal{T}) = \{ \mu \in \mathcal{F}(\mathcal{T}): \mu \le N \text{ a.s.} \}
	\end{align*}
	
	\subsection{Main results}
	
	Our first theorem generalises the Veraverbeke's result
	\cite{Verav1977} onto a fading BRW and provides the individual asymptotics for the tail distribution of the rightmost point of the fading BRW in all generations.
	
	\begin{theorem} \label{MR::th1}
		Let conditions \eqref{BPVE::non_ext}, \eqref{BPVE::fad_cond}, \eqref{BRW::zero_mean} and  \eqref{BRW::indep} hold. Assume that $ F_{I} \in \mathcal{S} $. Then, for any $ c > 0 $, 
		\begin{align*}
			\mathbb{P} \left( R_{\infty}^{\widehat{c}} > x \right) \sim H_{\infty}^{\widehat{c}} (x; \widehat{\mathcal{P}}) \sim  \frac{\mathbb{E} Z}{c} \cdot \overline{F}_{I} (x).
		\end{align*}
	\end{theorem}
	
We present now the individual asymptotics for the tail distribution of the rightmost point of the BRW within a random number of generations 	
	 $ \mu < \infty $ a.s.
	Let
	\begin{align*}
		\eta_{\mu} = \sum_{n=1}^{\mu} Z_{n}.
	\end{align*}
	The following result takes place.
	\begin{theorem} \label{MR::th2}
		Let conditions \eqref{BPVE::non_ext}, \eqref{BRW::zero_mean} and  \eqref{BRW::indep} hold and let conditions  {\normalfont{(HM)}} and $ \mathbb{E} \left( \mu Z_{\mu} \right) < \infty $ hold for a  random time $\mu$. Assume that $ F \in \mathcal{S^{*}} $. Then, for any $ c > 0 $,
		\begin{align}\label{th2::equiv1}
			\mathbb{P} \left( R_{\mu}^{\widehat{c}} > x \right) \sim H_{\mu}^{\widehat{c}} (x; \widehat{\mathcal{P}}) \sim  \mathbb{E} \eta_{\mu} \cdot \overline{F} (x).
		\end{align}
	\end{theorem}

	\begin{remark} \label{MR::r1}
	If we keep conditions \eqref{BPVE::non_ext}, \eqref{BRW::zero_mean}, \eqref{BRW::indep} and 
	(HM) of Theorem \ref{MR::th2} and replace condition $ \mathbb{E} \left( \mu Z_{\mu} \right) < \infty $ by an essentially stronger condition:
	$ \mathbb{E} Z_{\mu} (1+\delta)^{\mu} < \infty $ for some $ \delta > 0 $, then the equivalences 
	   \eqref{th2::equiv1} may be proved for a more general class of distributions $F\in \mathcal{S}$ and for any
	  $ c \in \mathbb{R} $. The proof of this result repeats the proof of Theorem \ref{MR::th2}, with replacing the reference to \cite[Theorem 2]{DenFossKor2010} by the reference to the Kesten's Lemma (see, e.g., \cite[Theorem 3.34]{FossKorZach2013}).
	\end{remark}

	\begin{remark}
		Assume there is no branching (i.e. 
		 $ \zeta_{n,1} =1 $ a.s. for all $ n \ge 0 $). Then Theorem
		 \ref{MR::th2} reduces to \eqref{Intro::FZ_res} for all $ \nu $ that do not depend on
		 the sequence $ \{ \xi_{i} \}_{i \ge 1} $.
	\end{remark}
	
	The following theorem extends onto fading BRW's the result of Proposition \ref{Intro::old_th1} for bounded
	stopping times.
	\begin{theorem} \label{MR::th3}
		Assume that conditions \eqref{BPVE::non_ext}, \eqref{BPVE::finite_mean}, \eqref{BRW::zero_mean} and \eqref{BRW::indep} hold.
		\begin{enumerate}
			\item[(i)] Assume that $ F \in \mathcal{L} $. Then, for any integer $ N \ge 1 $,  the inequality \eqref{Intro::goal_un1} holds uniformly in all $ \mu \in \mathcal{F}_{N}(\mathcal{T}) $ 
			and all $ g \in \mathcal{G}_{0} $.
			\item[(ii)] Assume that $ F \in \mathcal{S} $. then, for any integer $ N\ge 1 $,  the equivalence \eqref{Intro::goal_un2} holds uniformaly in all $ \mu \in \mathcal{F}_N(\mathcal{T}) $ and all $ g \in \mathcal{G}_0 $.
		\end{enumerate}
	\end{theorem}
	
	\section{Corollaries and Examples} \label{Examples}
		
		Let us formulate a corollary from Proposition \ref{BPVE::Z_moments}.
		\begin{corollary} \label{Examples::c1}
			Assume that conditions \eqref{BPVE::non_ext} and \eqref{BPVE::fad_cond} hold. Assume that there exists a constant $ K \ge 1 $ such that
			\begin{align} \label{c1::bound}
				\zeta_{n,1} \le K \quad \text{a.s.}
			\end{align}
			Then
			\begin{align} \label{c1::goal}
				\mathbb{E} Z^{s} < \infty
			\end{align}
			for all $ s > 0 $.
		\end{corollary}
		\begin{proof}[Proof]
			It is clear that, for any $ s > 0 $,
			\begin{align*}
				\prod_{n=0}^{\infty} \mathbb{E} \zeta_{n,1}^{s} \le \prod_{n=0}^{\infty} (1 + (K^{s}-1)q_{n}).
			\end{align*}
			The latter product is finite since the series  $ \sum_{n=0}^{\infty} q_{n} $ converges. 
			This and Proposition \ref{BPVE::Z_moments} imply \eqref{c1::goal}.
		\end{proof}

		Provide now a simple corollary from Theorem \ref{MR::th2} in the case where  $ \mu $ does not depend both on the branching process $ Z_{n} $ and on the increments of the BRW. In this particular case, we may determine the mean $ \mathbb{E} \eta_{\mu} $.
		\begin{corollary} \label{Examples::c2}
		Let conditions  \eqref{BPVE::non_ext}, \eqref{BPVE::finite_mean}, \eqref{BRW::zero_mean} and \eqref{BRW::indep} hold and assume that $ \mu $ does not depend on the $ \sigma $-algebra $ \sigma(\zeta_{n, j}, \xi_{i,k}; \ j,n,k,i \ge 0) $. Suppose that $ F \in \mathcal{S^{*}} $. 
		Then, for any $ c > 0 $,
		\begin{align*}
			\mathbb{P} \left( R^{\widehat{c}}_{\mu} > x \right) \sim \mathbb{E} \eta_{\mu} \overline{F} (x).
		\end{align*}
		Here 
		\begin{align*}
			\mathbb{E} \eta_{\mu} = \sum_{n=1}^{\infty} \mathbb{E}{Z_{n}} \mathbb{P} \left( \mu \ge n\right).
		\end{align*}
	\end{corollary}
	
	The next corollary is close in spirit to Remark \ref{MR::r1}. We consider conditions under which the equivalence \eqref{Intro::goal_non_un2} continues to hold if the drift of the BRW is positive, i.e. for  $ c > 0 $. 
	\begin{corollary} \label{Examples::htail}
	Let conditions \eqref{BPVE::non_ext}, \eqref{BPVE::finite_mean}, \eqref{BRW::indep}, \eqref{BRW::zero_mean} hold and let $ \mu $ satisfy condition (HM). Further, let $ \mathbb{E}{\mu Z_{\mu}} < \infty $, $ F \in \mathcal{S^{*}} $, and assume that, for a function $ h(x) $ such that 
	$ F $ is $h$-insensitive, we have 
	\begin{align*}
		\mathbb{P} \left( \mu > h(x) \right) = o \left( \overline{F} (x) \right).
	\end{align*}
	Then, for any $ c \in \mathbb{R} $,
	\begin{align} \label{htail::goal}
		\mathbb{P} \left( R^{\widehat{c}}_{\mu} > x \right) \sim \mathbb{E} \eta_{\mu} \cdot \overline{F} (x).
	\end{align}
	\end{corollary}
	\begin{proof}[Proof]
		Note that \eqref{htail::goal} follows directly from \ref{MR::th2} if $c>0$. Assume now $ c \le 0 $. Take any $ n \ge 1 $ and let $ \mu_{n} = \mu \wedge n $.
		
		Then, by Remark \ref{MR::r1},
		\begin{align} \label{htail::temp1}
			\mathbb{P} \left( R^{\widehat{c}}_{\mu_{n}}> x \right) \sim \mathbb{E} \eta_{\mu_{n}} \overline{F} (x).
		\end{align}
		Since $ R_{\mu}^{\widehat{c}} \ge R_{\mu_{n}}^{\widehat{c}}$ for al $ n $, the desired lower bound follows from  \eqref{htail::temp1}.
		
		Further, take an arbitrary $ a > |c| $. Then, for any $ n \ge 1 $,
		\begin{align*}
			R^{\widehat{c}}_{n} \le R^{\widehat{c}+\widehat{a}}_{n} + an 
		\end{align*}
		and, therefore, 
		\begin{align} \label{htail::temp2}
			\mathbb{P} \left( R^{\widehat{c}}_{\mu}>x \right) & \le \mathbb{P} \left( R^{\widehat{c} + \widehat{a}}_{\mu} + a \mu >x \right) \nonumber \\ 
			&= \mathbb{P} \left( R^{\widehat{a} + \widehat{c}}_{\mu} + a \mu >x, \mu < h(x) \right) + o(\overline{F}(x)) \nonumber \\
			& \le \mathbb{P} \left( R_{\mu}^{\widehat{a+c}}>x - ah(x) \right) + o(\overline{F}(x)).
		\end{align}
		
		Since $ a + c >0 $, by Theorem \ref{MR::th2}, 
		\begin{align} \label{htail::temp3}
			\mathbb{P} \left( R^{\widehat{a+c}}_{\mu}>x-ah(x) \right) \sim \mathbb{E} \eta_{\mu} \overline{F}(x-ah(x)) \sim \mathbb{E} \eta_{\mu} \overline{F} (x),
		\end{align}
		where the latter equivalence follows from the $ h $-insensitivity of distribution $ F $.
		
		Finally, the desired upper bound follows from \eqref{htail::temp2} and \eqref{htail::temp3}. 
	\end{proof}
	
	Consider now the case where  $ \mu $ is the fading time of the branching process,
	\begin{example} \label{Examples::ex1}
		Let conditions \eqref{BPVE::non_ext} \eqref{BPVE::finite_mean}, \eqref{BRW::indep} and  \eqref{BRW::zero_mean} hold. Assume that $ F \in \mathcal{S}^{*} $. Assume that $ \mu $ coincides with the fading time $ \nu $. Then, by Theorem \ref{MR::th2}, for any $ c > 0 $, 
		\begin{align} \label{ex1::temp1}
			\mathbb{P} \left( R_{\nu}^{\widehat{c}} > x \right) \sim \mathbb{E} \eta_{\nu} \overline{F}(x),
		\end{align}
		if we assume that  
		\begin{align} \label{ex1::temp2} 
			\mathbb{E} \nu Z_{\nu} \equiv \mathbb{E} \nu Z < \infty.
		\end{align}
		
		By the Holder inequality, \eqref{ex1::temp2} follows from the finiteness of the power moments $ \mathbb{E} \nu^{r} $ and $ \mathbb{E} Z^{s} $ for $ r,s > 0$ such that $ 1/r+1/s=1 $. In turn, Propositions \ref{BPVE::nu_moments} and \ref{BPVE::Z_moments} provide sufficient (and also neccessary for $ \mathbb{E} \nu^{r} $) conditions for the finiteness of these power moments. For example, in the case where
		\begin{align*}
			\sum_{n=1}^{\infty} n^{r} q_{n} < \infty
		\end{align*}
		and \eqref{c1::bound} holds, the equivalence \eqref{ex1::temp1} takes place too. 
	\end{example}

	Here is a simple example of the so-called ``intermediate'' asymptotics  of the tail distribution of interest, i.e. where the tail $ \mathbb{P} \left( R_{\mu}^{g} > x \right) $  appears to be heavier than $ \overline{F} (x) $ but lighter than $ \overline{F}_{I} (x) $.
	\begin{example} \label{Examples::ex2}
		 Assume that $ \zeta_{0,1} = 2 $ a.s. and $ \zeta_{1,1} = \zeta_{2,1} = \zeta_{3,1} = \ldots = 1 $ a.s. Here $ \nu = 1 $ and $ Z = 2 $. Assume also that conditions  \eqref{BRW::zero_mean} and \eqref{BRW::indep} hold and that
			$ \overline{F} (x) \sim K_{1} x^{-\beta}$, 
		for some $ \beta > 1 $.
		
		Consider a random time $ \mu $ that satisfies condition (HM) and has the distribution tail $\mathbb{P} \left( \mu \ge n \right) \sim K_{2} n^{-\alpha} $, $ n \to \infty $ for some $ \alpha \in (0,1) $ (then, in particular,  $ \mathbb{E} \mu = \infty$).
		Then, for any $ c > 0 $,
		\begin{align}
			\mathbb{P} \left( R_{\mu}^{\widehat{c}} > x \right) & \sim  \mathbb{P} \left( M_{1, \mu}^{\widehat{c}} > x \right) + \mathbb{P} \left( M_{2, \mu}^{\widehat{c}} > x \right) \nonumber \\
			& \sim 2 H_{\mu}^{\widehat{c}} (x). \label{ex2::temp2}
		\end{align}
		Here $ M_{i, \mu}^{\widehat{c}} = \max_{\pi \subseteq \pi_{i, \mu}} S^{\widehat{c}}(\pi) $, where $ \pi_{1, \mu} $ and $ \pi_{2, \mu} $ are different paths of length $ \mu $ in the genealogical tree of the branching process and $ \pi $ a path that starts from its root. The last equivalence  \eqref{ex2::temp2} follows from \eqref{Intro::old_th2}.
		
		Further, 
		\begin{align}
			H_{\mu}^{\widehat{c}} (x) &= \sum_{n=1}^{\infty} \mathbb{P} \left( \mu \ge n \right) \overline{F} (x + cn) \sim K_{1} K_{2}  \sum_{n=1}^{\infty} n^{-\alpha} (x+cn)^{-\beta} \nonumber \\
			& \sim K_{1} K_{2}  \int_{0}^{\infty} t^{-\alpha} (x+ct)^{-\beta} dt \label{ex2::temp1} \\
			& \sim C x^{1- \alpha - \beta}, \nonumber
		\end{align}
	where $ C = K_1 K_2 \mathbb{E}Z c^{\alpha-1} B (1-\alpha, \beta + \alpha - 1) $ and the equivalence  \eqref{ex2::temp1} holds since condition $ \mathbb{E} \mu = \infty $ implies that any finite number of summands in the presentation for function  $ H_{\mu}^{g} (x) $ does not bring an essential contribution as $ x \to \infty $, i.e. $ \sum_{n=1}^{N} n^{-\alpha} (x+cn)^{-\beta} = O(x^{-\beta}) = o(x^{1-\alpha-\beta}) $, for any finite $ N > 1 $.
	
	Therefore,
	\begin{align*}
		\mathbb{P} \left( R_{\mu}^{\widehat{c}} > x \right) \sim 2C x^{1-\alpha-\beta}
	\end{align*}
	and, in particular, $ \overline{F} (x) = o(\mathbb{P} ( R_{\mu}^{\widehat{c}} > x)) $ and 
	$\mathbb{P} ( R_{\mu}^{\widehat{c}} > x) = o( \overline{F^I} (x)) $ as $x\to\infty$.
	\end{example}
	
	\section{Proofs of the main results} \label{proofs}
	
	\begin{proof}[Proof of Theorem \ref{MR::th1}]
		Our branching process includes $Z$ infinite paths that start from the root of the genealogical tree $ \mathcal{T} $. We enumerate these paths arbitrarily and, for  $ i = 1, 2, \ldots, Z $, denote by $ S_{i,n}, n \ge 0 $  the random walk along the $i$'th path and by $ M_{i}^{\widehat{c}} $ the maximum of the corresponding $\widehat{c}$-shifted random walk on the infinite time horizon. Then $ R_{\infty}^{\widehat{c}} = \max_{1 \le i \le Z}  M_{i}^{\widehat{c}} $ and
		\begin{align} \label{th1::temp1}
			\mathbb{P} \left( R_{\infty}^{\widehat{c}} > x \right) &= \mathbb{P} \left(\max_{1 \le i \le Z} M_{i}^{\widehat{c}} > x \right) = \sum_{m=1}^{\infty} \mathbb{P} \left( \max_{1 \le i \le m} M_{i}^{\widehat{c}} > x, Z = m \right) \nonumber \\
			& \le \sum_{m=1}^{\infty} m \mathbb{P} \left( M_{1}^{\widehat{c}} > x \right) \mathbb{P} \left( Z = m \right) = \mathbb{E} Z \cdot \mathbb{P} \left(M_{1}^{\widehat{c}} > x \right).
		\end{align}
		
		Then \eqref{th1::temp1} and \eqref{Intro::Veraver_res} imply that 
		\begin{align} \label{th1::temp2}
			\limsup_{x \to \infty} \frac{1}{\overline{F}_{I}(x)} \cdot \mathbb{P} \left( R_{\infty}^{\widehat{c}} > x \right) \le \frac{\mathbb{E} Z}{c}.
		\end{align}
		
		We obtain now a similar to \eqref{th1::temp2} lower bound. Introduce new random walks 
		 $S_{i,n}^{\prime}$ using the following time shift: for any  $ i = 1, \dots, Z $, we let 
		\begin{align*}
			S_{i,n}^{\nu} = S_{i, \nu + n} - S_{i, \nu}, \quad n \ge 0.
		\end{align*}
		Introduce also
		\begin{align*}
			S_{i, n}^{\widehat{c}, \nu} = S_{i,n}^{\nu} - cn, \qquad n \ge 0,
		\end{align*}
		the corresponding family of $\widehat{c}$-shifted random walks, and let 
		\begin{align*}
			M_{i}^{\widehat{c}, \nu} = \sup_{n \ge 0} S_{i,n}^{\widehat{c}, \nu}.
		\end{align*}
		
		Then, for any integer $N$, 
		\begin{align*}
			\mathbb{P} \left( R_{\infty}^{\widehat{c}}>x \right) & \ge \mathbb{P} \left( \max_{1 \le i \le Z} M_{i}^{\widehat{c}, \nu} > x + H(x), l_{\nu}^{\widehat{c}}> - H(x)\right) \nonumber \\
			&\ge \sum_{n,m=0}^{N} \mathbb{P} \left(\max_{1 \le i \le m} M_{i}^{\widehat{c}, n} > x + H(x), l_{n}^{\widehat{c}} > - H(x) \right) \mathbb{P} \left( \nu = n, Z = m \right),
		\end{align*}
		where we assume that the distribution $ F_{I} $ is $ H $-insensitive (note that such a function $H$ exists since $ \mathcal{S} \subset \mathcal{L} $) and where $ l_{\nu}^{\widehat{c}} = \inf_{\pi : |\pi| = \nu} S^{\widehat{c}}(\pi) $ is the leftmost point of the BRW in the $\nu$'th generation.
		
		For all $ 0 \le n \le N $ and $ 1 \le m \le N $ the random vector $ (M_{1}^{\widehat{c}, n}, \ldots, M_{m}^{\widehat{c}, n}) $ does not depend on  $ l_{n}^{\widehat{c}} $ and, therefore,
		\begin{align*}
		 	\mathbb{P} \left( R_{\infty}^{\widehat{c}}>x \right) & \ge \sum_{n,m=0}^{N} \mathbb{P} \left(\max_{1 \le i \le m} M_{i}^{\widehat{c}, n} > x + H(x) \right) \mathbb{P} \left(l_{n}^{\widehat{c}} > - H(x) \right) \mathbb{P} \left( \nu = n, Z = m \right) \nonumber
			\\ & \ge  \mathbb{P} \left( l_{N}^{\widehat{c}} > - H(x) \right) \sum_{n,m=0}^{N} \mathbb{P} \left( \max_{1 \le i \le m} M_{i}^{\widehat{c}, n} > x + H(x) \right) \mathbb{P} \left( \nu = n, Z = m \right) \nonumber
			\\ & = (1+o(1)) \sum_{n,m=0}^{N} \mathbb{P} \left( \max_{1 \le i \le m} M_{i}^{\widehat{c}, n} > x + H(x) \right) \mathbb{P} \left( \nu = n, Z = m \right).
		\end{align*}
		
		Since the r.v.'s $ M_{1}^{\widehat{c}, n}, \ldots, M_{m}^{\widehat{c}, n} $ are mutually independent and have a common subexponential distribution,
		\begin{align*}
			\mathbb{P} \left( \max_{1 \le i \le m} M_{i}^{\widehat{c}, n} > x + H(x) \right) \sim m \mathbb{P} \left( M_{1}^{\widehat{c}, n} > x + H(x) \right) \sim \frac{m}{c} \cdot \overline{F}_{I} (x),
		\end{align*}
		for all $ m=1,\ldots,N $, where the  latter equivalence follows from the  $ H $-insensitivity of function $ \overline{F}_{I} (x) $ and from \eqref{Intro::Veraver_res}.
		Therefore,
		\begin{align*}
			\liminf_{x \to \infty} \frac{1}{\overline{F}_{I}(x)} \mathbb{P} \left( R_{\infty}^{\widehat{c}}>x \right) \ge \frac{1}{c} \sum_{n,m=0}^{N} m \mathbb{P} \left( \nu = n, Z = m \right),
		\end{align*}
		for all $ N \ge 1 $. Letting $ N $ tend to infinity, we get 
		\begin{align} \label{th1::temp3}
			\liminf_{x \to \infty} \frac{1}{\overline{F}_{I}(x)} \mathbb{P} \left( R_{\infty}^{\widehat{c}}>x \right) \ge \frac{\mathbb{E} Z}{c}.
		\end{align}
		Inequalities \eqref{th1::temp2} and \eqref{th1::temp3} lead to the desired result.
		
		Theorem \ref{MR::th1} is proved.
	\end{proof}

	\begin{proof}[Proof of Theorem \ref{MR::th2}]
			The total probability formula leads to 
			\begin{align*}
				\frac{1}{\overline{F}(x)} \mathbb{P} \left( R^{\widehat{c}}_{\mu} > x \right) = \sum_{n,m,k=1}^{\infty} \frac{1}{\overline{F}(x)} \mathbb{P} \left(R^{\widehat{c}}_{\mu} > x, \mu = n, Z_{\mu} = m, \eta_{\mu} = k \right)
			\end{align*}
			
			Consider any of the summands on the right-hald side. Let 
			\begin{align*}
				S_{\eta_{\mu}}^{+} \equiv S_{\eta_{\mu}}^{+} (\mathcal{T})  = \sum_{e \in \mathcal{E}_{\mu}} \xi_{e}^{+},
			\end{align*}
			where the set $ \mathcal{E}_{\mu} = \{ e_{1}, \ldots, e_{\eta_{\mu}} \} $ consists of the edges that connect vertices of generations  $n\le \mu $. Clearly, $ R^{\widehat{c}}_{\mu} \le S_{\eta_{\mu}}^{+} $. Then 
			\begin{align} \label{th2::temp1}
				\mathbb{P} \left(R^{\widehat{c}}_{\mu} > x, \mu = n, Z_{\mu} = m, \eta_\mu = k \right) &\le \mathbb{P} \left( 
				S_{\eta_\mu}^{+} > x, \mu = n, Z_{\mu} = m, \eta_\mu = k \right) \nonumber \\
				& = \mathbb{P} \left( S_{k}^{+} > x \right) \mathbb{P} \left( \mu = n, Z_{\mu} = m, \eta_\mu = k 
				\right) \nonumber \\
				& \sim  k \overline{F} (x) \cdot \mathbb{P} \left( \mu = n, Z_{\mu} = m, \eta_\mu = k \right),
			\end{align}
			where the equivalence above follows from  $ F \in \mathcal{S} $. 
			
			For the lower bound, we need again a function $h(x)\uparrow \infty$ such that $F$ is
			$h$-insensitive. For $ i = 1, \ldots, \eta_\mu $,  let
			\begin{align*}
			A_{i} (x, \eta_\mu) = \{ \xi_{e_{i}} -c > x + \eta_\mu h(x) \} \cap \bigcap_{j \neq i} \{ - h(x) < \xi_{e_{j}} - c < h(x) \}.
			\end{align*}
			Then 
			\begin{align} \label{th2::temp2}
				\mathbb{P} \left( R^{\widehat{c}}_{\mu} > x, \mu = n, Z_{\mu} = m, \eta_{\mu} = k \right) &\ge  \mathbb{P} \left( \bigcup_{i=1}^{\eta_{\mu}} A_{i}(x, \eta_\mu), \mu = n, Z_{\mu} = m, \eta_{\mu} = k \right) \nonumber \\
				& =  \mathbb{P} \left( \bigcup_{i=1}^{k} A_{i}(x, k), \mu = n, Z_{\mu} = m, \eta_{\mu} = k \right) \nonumber \\
				& =  \sum_{i=1}^{k} \mathbb{P} \left( A_{i}(x,k) \right) \mathbb{P} \left( \mu = n, Z_{\mu} = m, \eta_{\mu} = k \right) \nonumber \\
				& \sim k \overline{F} (x) \cdot \mathbb{P} \left( \mu = n, Z_{\mu} = m, \eta_{\mu} = k \right),
			\end{align}
			where the latter equivalence follows from the $ h $-insensitivity of the function $ \overline{F} (x) $ and from equalities $ \mathbb{P} \left( A_{i}(x,k) \right) = \overline{F} (x+kh(x)+c) (\mathbb{P} \left( |\xi-c|<h(x)\right))^{k-1} $, for all $ i=1,\ldots,k $.
			
			It follows from \eqref{th2::temp1} and \eqref{th2::temp2} that, for any $ n \ge 1, m \ge 1, k \ge 1 $, the following equivalence holds: 
			\begin{align*}
				\mathbb{P} \left( R^{\widehat{c}}_{\mu} > x, \mu = n, Z_{\mu} = m, \eta_\mu = k \right) \sim k \overline{F} (x) \cdot \mathbb{P} \left( \mu = n, Z_{\mu} = m, \eta_\mu = k \right).
			\end{align*}
			
			Denote by $ \pi_{i}, \ i=1, 2, \ldots, Z_{\mu} $ all paths of length $ \mu $ that start from the root of the genealogical tree, and let
			\begin{align*}
				M^{\widehat{c}}_{i, \mu} = \text{sup}_{i} S^{\widehat{c}}(\pi),
			\end{align*}
			where the $ i $'th supremum is taken over all sub-paths $ \pi $ of path $ \pi_{i} $ that start from the root. Then 
			\begin{align} \label{th2::temp3}
				\mathbb{P} \left( R^{\widehat{c}}_{\mu} > x, \mu = n, Z_{\mu} = m, \eta_\mu = k \right) & = \mathbb{P} \left(\max_{1 \le i \le Z_{\mu}} M^{\widehat{c}}_{i, \mu} > x, \mu = n, Z_{\mu} = m, \eta_\mu = k \right) \nonumber \\
				& = \mathbb{P} \left( \max_{1 \le i \le m} M^{\widehat{c}}_{i, n} > x, \mu = n, Z_{\mu} = m, \eta_\mu = k \right) \nonumber \\
				& \le  m \mathbb{P} \left(  M^{\widehat{c}}_{1, n} > x \right) \mathbb{P} \left( \mu = n, Z_{\mu} = m, \eta_\mu = k\right).
			\end{align}
			
			By \cite[
			Theorem 2]{DenFossKor2010},
			\begin{align} \label{th2::temp4}
				\mathbb{P} \left(M^{\widehat{c}}_{1,n} > x \right) \le C n \overline{F} (x), \text{for all} x>0, n \ge 1. 
			\end{align}
			
			Using \eqref{th2::temp4}, we may continue the chain of inequalities  \eqref{th2::temp3} as follows:
			\begin{align*}
				\frac{1}{\overline{F}(x)} \mathbb{P} \left( R^{\widehat{c}}_{\mu} > x, \mu = n, Z_{\mu} = m, \eta_\mu = k \right) \le C \cdot m \cdot n \cdot \mathbb{P} \left( \mu = n, Z_{\mu} = m, \eta_\mu = k \right).
			\end{align*}
			Since
			\begin{align*}
				\sum_{n,m,k=1}^{\infty} m n \mathbb{P} \left( \mu = n, Z_{\mu} = m, \eta_\mu = k \right) = \mathbb{E}\mu Z_{\mu} < \infty,
			\end{align*}
			we can apply the Lebesgue's dominated convergence theorem to get
			\begin{align*}
				&\lim_{x \to \infty} \sum_{n,m,k=1}^{\infty}  \frac{1}{\overline{F}(x)} \mathbb{P} \left( R^{\widehat{c}}_{\mu} > x, \mu = n, Z_{\mu} = m, \eta_\mu = k \right) \\ =& \sum_{n,m,k=1}^{\infty}  k \mathbb{P} \left( \mu = n, Z_{\mu} = m, \eta_\mu = k \right) = \mathbb{E} \eta_\mu.
			\end{align*}
			Since $ \eta_\mu \le \mu Z_{\mu} $ a.s., we have $ \mathbb{E} \eta_\mu < \infty $.
			Therefore,
			\begin{align*}
				\mathbb{P} \left( R^{\widehat{c}}_{\mu} > x \right) \sim \mathbb{E} \eta_\mu \overline{F} (x).
			\end{align*}
			
			Theorem \ref{MR::th2} is proved.
	\end{proof}

	\begin{proof}[Proof of Theorem \ref{MR::th3}]
		Let
		\begin{align*}
			r_{n}^{g} = \max_{\pi: |\pi|=n} S^{g}(\pi), \ \ l_{n}^{g} = \min_{\pi: |\pi|=n} S^{g}(\pi)
		\end{align*}
		and
		\begin{align*}
			\tau^{g}(x) = \inf \{ n \ge 1: r_{n}^{g} > x \}\le \infty.
		\end{align*}
		
		Let a function $ h(x) \to \infty $ be such that the distribution $ F $ is $ h $-insensitive. Then, for any $ \mu \in \mathcal{F}_{N}(\mathcal{T}) $, 
		\begin{align} \label{th3::decomp}
			\mathbb{P} \left( R_{\mu}^{g} > x \right) &= \mathbb{P} \left( \mu \ge \tau^{g}(x) \right) = \sum_{n=1}^{N} \mathbb{P} \left( \mu \ge n, r_{n}^{g}>x, R^{g}_{n-1} \le x \right) \nonumber \\
			&= \sum_{n=1}^{N} \mathbb{P} \left( \mu \ge n, r_{n}^{g} > x, -h(x+g(n-1)) \le l_{n-1}, r_{n-1} \le h(x+g(n-1)), R_{n-2}^{g} \le x \right) \nonumber \\
			&+ \sum_{n=1}^{N} \mathbb{P} \left( \mu \ge n, r_{n}^{g} > x, l_{n-1} < -h(x+g(n-1)), r_{n-1} \le h(x+g(n-1)), R_{n-2}^{g} \le x \right) \nonumber \\
			& + \sum_{n=1}^{N} \mathbb{P} \left( \mu \ge n, r_{n}^{g}>x, r_{n-1} > h(x+g(n-1)), R_{n-2}^{g} \le x\right) \nonumber \\
			&=: \sum_{n=1}^{N} a_{\mu,1}^{g}(x,n) + \sum_{n=1}^{N} a_{\mu,2}^{g}(x,n) + \sum_{n=1}^{N} a_{\mu,3}^{g}(x,n).
		\end{align}
		
		We start with the proof of the first statement of the theorem. Since $ \mu $ does not depend on the future, 
		\begin{align*}
			a_{\mu,1}^{g}(x,n) \ge \sum_{k=1}^{K} \mathbb{P} \left( A_{\mu}^{g}(n,k,x) \right) [ 1 - (1- \overline{F} (x+g(n)+h(x+g(n-1))))^{k} ]
		\end{align*}
		for all integer $ K \ge 1 $, where
		\begin{align*}
			A_{\mu}^{g}(n,k,x) = \{ \mu \ge n, Z_{n} = k, -h(x+g(n-1)) \le l_{n-1}, r_{n-1} \le h(x+g(n-1)), R_{n-2}^{g} \le x \}.
		\end{align*}
		
		Since for $ x \to \infty $
		\begin{align*}
			\mathbb{P} \left( R_{n-2}^{g} > x \right) \le \mathbb{P} \left( R_{n-2} > x \right) \to 0
		\end{align*}
		and
		\begin{align*}
			\mathbb{P} \left( l_{n-1} < -h(x+g(n-1)) \right) &+ \mathbb{P} \left(r_{n-1} > h(x + g(n-1)) \right) \\ &\le \mathbb{P} \left( l_{n-1} < - h(x) \right) + \mathbb{P} \left( r_{n-1} > h(x) \right) \to 0,
		\end{align*}
		we get that 
		\begin{align} \label{th3::lb1}
			\mathbb{P} \left( A_{\mu}^{g}(n,k,x) \right) = \mathbb{P} \left( \mu \ge n, Z_{n} = k \right) + o(1)
		\end{align}
		uniformly in all $ g \in \mathcal{G}_{0} $ and $ \mu \in \mathcal{F}_{N} (\mathcal{T}) $.
		
		Since the distribution $ F $ is $ h $-insensitive,
		\begin{align} \label{th3::mon_un_eq}
			\sup_{y\ge x} \left| \frac{\overline{F}(y+h(y))}{\overline{F}(y)} - 1 \right| \to 0
		\end{align}
		as $ x \to \infty $. Then it follows from \eqref{th3::mon_un_eq} and from the inequality  $g(n-1)\le g(n)$ that 
		\begin{align} \label{th3::kr}
			\overline{F} (x+g(n) \pm h(x+g(n-1))) = (1+o(1)) \overline{F} (x+g(n))
		\end{align}
		uniformly in all $ g \in \mathcal{G}_{0} $.
		
		Since $ 1 - (1 - \overline{F}(x))^k \sim k \overline{F} (x) $ and $ g \in \mathcal{G}_{0} $, it follows from \eqref{th3::kr} that 
		\begin{align} \label{th3::lb2}
			1 - (1- \overline{F} (x+g(n)+h(x+g(n-1))))^{k} = (1 + o(1)) k \overline{F} (x + g(n))
		\end{align}
		uniformly in all $ g \in \mathcal{G}_{0} $.
		
		From \eqref{th3::lb1} and \eqref{th3::lb2},
		\begin{align*}
			\sum_{n=1}^{N}a_{\mu,1}^{g}(x,n) &\ge (1+o(1)) \sum_{n=1}^{N} \sum_{k=1}^{K} k [\mathbb{P} \left( \mu \ge n, Z_{n} = k \right) + o(1)] \overline{F} (x + g(n)) \\ 
			&= (1 + o(1)) \sum_{n=1}^{N} \mathbb{E} \left[ Z_n \mathbb{I} \left( \mu \ge n \right) \left(Z_{n} \le K \right) \right] \overline{F} (x + g(n)) + o(\overline{F} (x + g(1))) \\
			& = (1 + o(1)) \sum_{n=1}^{N} \mathbb{E} \left[ Z_{n} \mathbb{I} \left( \mu \ge n \right) \mathbb{I} \left( Z_{n} \le K \right) \right] \overline{F} (x + g(n)).
		\end{align*}
		
		Further, let $ \varepsilon > 0 $. Since
		\begin{align*}
			\mathbb{E} \left[Z_{n}\mathbb{I} \left(\mu \ge n\right) \mathbb{I} \left( Z_{n} > K \right) \right] \le \mathbb{E} \left[ Z_{n} \mathbb{I} \left( Z_{n} > K \right) \right],
		\end{align*}
	we can choose a sufficiently large $ K $
		such that
		\begin{align*}
			\mathbb{E} \left[Z_{n} \mathbb{I} \left( \mu \ge n \right) \mathbb{I} \left( Z_{n} \le K \right) \right] \ge \mathbb{E} \left[ Z_{n} \mathbb{I} \left( \mu \ge n \right) \right] - \varepsilon,
		\end{align*}
		for all $ n \le N $. Therefore, 
		\begin{align}
			\sum_{n=1}^{N}a_{\mu,1}^{g} (x,n) &\ge (1+o(1)) H_{\mu}^{g} (x; \widehat{\mathcal{P}}) -(1+o(1)) \varepsilon \sum_{n=1}^{N} \overline{F} (x+g(n)) \nonumber \\
			& \ge (1+o(1)) H_{\mu}^{g} (x; \widehat{\mathcal{P}}) - (1+ o(1))N \cdot \varepsilon H_{\mu}^{g} (x; \widehat{\mathcal{P}}) \nonumber \\
			& = (1+o(1)) (1 - \varepsilon N) H_{\mu}^{g}(x; \widehat{\mathcal{P}}). \label{th3::lb3}
		\end{align}
		
		Since $\varepsilon >0$ is arbitrary, it follows from \eqref{th3::lb3} that
		\begin{align} \label{th3::lb4}
			\mathbb{P} \left( R_{\mu}^{g} > x \right) \ge \sum_{n=1}^{N} a_{\mu,1}^{g}(x,n) \ge (1+o(1)) H_{\mu}^{g}(x; \widehat{\mathcal{P}}),
		\end{align}
		uniformly in all $ \mu \in \mathcal{F}_{N}(\mathcal{T}) $ and $ g \in \mathcal{G}_{0} $.
		
		Now we prove the second statement of the theorem. For that, we show that the first summand in  \eqref{th3::decomp} has the right asymptotics while the second and the third summands are negligible.
		
		Clearly,
		\begin{align*}
			a_{\mu,1}^{g}(x,n) & \le \sum_{k=1}^{\infty} \mathbb{P} \left( \mu \ge n, \max_{1\le j \le k} \xi_{n,j} > x +g(n) -h(x+g(n-1)), Z_n = k \right) \\
			& \le \sum_{k=1}^{\infty} k \mathbb{P} \left( \mu \ge n, Z_n = k \right) \overline{F} (x +g(n) - h(x + g(n-1))),
		\end{align*}
		where the second inequality holds since $ \mu $ does not depend on the future. Therefore, by
		 \eqref{th3::kr},
		\begin{align*}
			a_{\mu,1}^{g}(x,n) \le (1+o(1)) \mathbb{E} \left[ Z_n \mathbb{I} (\mu \ge n) \right] \overline{F} (x+g(n))
			& \le (1+o(1)) \mathbb{E} \left[ Z_n \mathbb{I} (\mu \ge n) \right] \overline{F} (x+g(n))
		\end{align*}
		and
		\begin{align} \label{th3::ub1}
			\sum_{n=1}^{N}a_{\mu,1}^{g}(x,n) \le (1+o(1)) H_{\mu}^{g}(x; \widehat{\mathcal{P}}),
		\end{align}
		where the summand $ o(1) $ is small uniformly in all $ g \in \mathcal{G}_{0} $ and $ \mu \in \mathcal{F}_{N} (\mathcal{T}) $. 
		
		Taking into account \eqref{th3::lb4} and \eqref{th3::ub1}, we get that 
		\begin{align} \label{th3::ub2}
			\sum_{n=1}^{N}a_{\mu,1}^{g} (x) = (1 + o(1)) H_{\mu}^{g} (x; \widehat{\mathcal{P}}),
		\end{align}
		uniformly in all $ \mu \in \mathcal{F}_{N}(\mathcal{T}) $ and $ g \in \mathcal{G}_{0} $.
		
		Further,
		\begin{align*}
			a_{\mu,2}^{g}(x, n) \le \sum_{k=1}^{\infty} \mathbb{P} \left( \mu \ge n, Z_{n}=k, l_{n} \le -h(x+g(n-1)) \right) \overline{F}(x+g(n) - h(x+g(n-1))),
		\end{align*}
		since $ \mu $ does not depend on the future. Then, by \eqref{th3::kr},
		\begin{align*}
			a_{\mu,2}^{g}(x,n) &\le (1+o(1)) \mathbb{E} \left[ Z_{n} \mathbb{I} \left( \mu \ge n \right) \mathbb{I} \left(l_{n}\le-h(x+g(n-1)) \right) \right] \overline{F}(x+g(n)) \\
			&\le (1+o(1)) \mathbb{E} \left[ Z_{n} \mathbb{I} \left( l_{n}\le-h(x) \right) \right] H_{\mu}^{g}(x; \widehat{\mathcal{P}}) \\
			& = o(H_{\mu}^{g}(x; \widehat{\mathcal{P}}))
		\end{align*}
		and, therefore,
		\begin{align} \label{th3::ub3}
			\sum_{n=1}^{N} a_{\mu,2}^{g}(x,n) = o(H_{\mu}^{g}(x; \widehat{\mathcal{P}})),
		\end{align}
		uniformly in all $ \mu \in \mathcal{F}_{N}(\mathcal{T}) $ and $ g \in \mathcal{G}_{0} $.
		
		Now consider the third summand in \eqref{th3::decomp}. 
		It follows from Theorem \ref{MR::th2} (for $ \mu = N $ a.s.) that the following equivalence holds: 
		\begin{align*}
			\mathbb{P} \left( R_{N}^{\widehat{0}} > x \right) \sim \mathbb{E} \eta_{N} \overline{F} (x).
		\end{align*}
		Since $ H_{N}^{0} (x; \widehat{\mathcal{P}}) = \mathbb{E} \eta_{N} \overline{F} (x) $, it follows from \eqref{th3::ub2} that 
		\begin{align*}
			\sum_{n=1}^{N} a_{N,3}^{0}(x,n) = o (H_{N}^{0}(x; \widehat{\mathcal{P}})).
		\end{align*}
		
		Next, for any $ \mu \in \mathcal{F}_{N}(\mathcal{T}) $ and $ g \in \mathcal{G}_{0} $,  
		\begin{align*}
			a_{\mu,3}^{g}(x,n) \le a_{N,3}^{0}(x+g(n-1),n),
		\end{align*}
		and since $ a_{\mu,3}^{g}(x,1) = 0 $, we get that
		\begin{align} \label{th3::ub4}
			\sum_{n=1}^{N} a_{\mu,3}^{g}(x,n) \le \sum_{n=2}^{N} a_{N,3}^{0} (x+g(n-1),n) = o(H_{\mu}^{g}(x; \widehat{\mathcal{P}}))
		\end{align}
		uniformly in all $ \mu \in \mathcal{F}_{N}(\mathcal{T}) $ and $ g \in \mathcal{G}_{0} $.
		
		Finally, \eqref{th3::ub2}, \eqref{th3::ub3} and \eqref{th3::ub4} imply the second statement of Theorem \ref{MR::th3}.
	\end{proof}
	
	\section*{Appendix} \label{app}
	
	\subsection*{A1. Proofs of Propositions \ref{BPVE::nu_moments} and \ref{BPVE::Z_moments}}
	For any $ n \ge 0 $, we let 
	\begin{align}
		d_{n} = - \sum_{k=n}^{\infty} \ln (1 - q_{k}).
	\end{align}
	
	\begin{lemma} \label{A1::nu_tail}
		Assume that \eqref{BPVE::fad_cond} holds. Then 
		\begin{align*}
			\mathbb{P}\left( \nu \le n \right) \in [ e^{-L d_{n}}, e^{-d_{n}} ].
		\end{align*}
	\end{lemma}
	\begin{proof}[Proof]
		Since
		\begin{align*}
			\mathbb{P}\left( \nu \le n \right) & = \mathbb{P}\left( Z_{k} = Z_{n} \text{ for all } k \ge n \right) \nonumber \\ 
			&= \mathbb{P} \left(\zeta_{k,j} = 1 \text{ for all } j = 1, \ldots, Z_{n} \text{ and } k \ge n \right) \nonumber \\
			& = \mathbb{E} \left[ \mathbb{P} \left(\zeta_{k,j} = 1 \text{ for all } j = 1, \ldots, Z_{n} \text{ and } k \ge n | Z_{n} \right) \right] \nonumber \\ 
			&= \mathbb{E} \left[ \prod_{k=n}^{\infty} (1-q_{k})^{Z_{n}} \right] \nonumber \\
			& = \mathbb{E} e^{-d_{n} Z_{n}}
		\end{align*}
		and $ Z_{n} \ge 1 $ a.s., we get 
		\begin{align*}
			\mathbb{P} \left( \nu \le n \right) \le e^{-d_{n}}.
		\end{align*}
			
		On the other hand, by the Jensen's inequality, 
		\begin{align*}
			\mathbb{P} \left( \nu \le n \right) \ge e^{-d_{n} \mathbb{E}Z_{n}}.
		\end{align*}
		Then \eqref{BPVE::fad_cond} implies that $ \mathbb{E} Z_{n} \le L < \infty $ and, therefore,
		\begin{align*}
			\mathbb{P}\left( \nu \le n \right) \ge e^{-L d_{n}}.
		\end{align*}
		
		Lemma \ref{A1::nu_tail} is proved.
	\end{proof}
	
	\begin{proof}[Proof of Proposition \ref{BPVE::nu_moments}]
		Without loss of generality, we may assume that  $ f(0) = f(0+0) = 0 $. Then  
		\begin{align} \label{nu_moments::temp1}
			\mathbb{E} f(\nu) = \int_{0}^{\infty} \mathbb{P} \left( \nu \ge t \right) df(t) = \sum_{n=0}^{\infty} \mathbb{P} \left( \nu > n\right) \cdot (f(n+1) - f(n)).
		\end{align}
		
		By Lemma \ref{A1::nu_tail}, the series in the RHS of   \eqref{nu_moments::temp1}
		converges if and only if
		\begin{align*}
			\sum_{n=0}^{\infty} (f(n+1) - f(n)) d_{n} <\infty.
		\end{align*}
		It follows from the equivalence $ d_n \sim \sum_{k=n}^{\infty} q_k $ that $ \mathbb{E} f(\nu) < \infty $ if and only if
		\begin{align*}
			\sum_{n=0}^{\infty} (f(n+1) - f(n)) \sum_{k=n}^{\infty} q_{k} = \sum_{k=0}^{\infty} \left( \sum_{n=0}^{k} (f(n+1) - f(n)) \right) q_k = \sum_{k=0}^{\infty} f(k+1) \cdot q_{k} < \infty.
		\end{align*}
		Proposition \ref{BPVE::nu_moments} is proved.
	\end{proof}
	
	\begin{proof}[Proof of Proposition \ref{BPVE::Z_moments}]
		We use the following well-known inequality:
		\begin{align*}
			\left| \sum_{k=1}^{n} x_{k} \right|^{r} \le \max \{1, n^{r-1} \} \sum_{k=1}^{n} |x_{k}|^{r} \text{ for all } r > 0 \text{ and for all } x_{1}, \ldots, x_{n} \in \mathbb{R}. 
		\end{align*}
		
		Applying this inequality to $ Z_{n}^{s} $, we get: 
		\begin{align*}
			\mathbb{E} (Z_{n})^s & \le \mathbb{E} \left[Z_{n-1}^{s-1} \sum_{j=1}^{Z_{n-1}} \zeta_{n-1,j}^{s} \right] = \mathbb{E} Z_{n-1}^{s} \mathbb{E} \zeta_{n-1,1}^s \le \prod_{k=0}^{n-1} \mathbb{E} \zeta_{k,1}^{s}.
		\end{align*}
		By the Fatou's Lemma,
		\begin{align*}
			\mathbb{E} Z^{s} \le \prod_{n=0}^{\infty} \mathbb{E} \zeta_{n,1}^{s} < \infty.
		\end{align*}
		Proposition \ref{BPVE::Z_moments} is proved.
	\end{proof}

	\subsection*{A2. Time-homogeneous BRW with heavy-tailed distribution of increments}
	
	The statement presented below (which does not pretend to be original) shows that the supremum over infinite time interval of a time-homogeneous BRW is always infinite, if one assumes that condition  
	\eqref{BPVE::non_ext} holds and that the common distribution of jumps has a heavy right tail. This fact shows the necessity of the condition of fading branching for the study of the supremum over infinite time horizon of BRW's with heavy-tailed jump distributions. 

	\begin{proposition}\label{infinity}
		Assume that \eqref{BPVE::non_ext} holds, that $\zeta_n$, $n=1,2,\ldots$ have a common distribution $G$, and that 
		$q\equiv q_n = {\mathbb P}(\zeta_{1,1}>1)>0$. Assume also that the distribution $F$ of the increments of the BRW has a heavy right tail. Then, for any $c \in \mathbb{R}$, 
		\begin{align*}
			R^{\widehat{c}} \equiv \sup_n R_n^{\widehat{c}}=\infty \quad \mbox{a.s.}
		\end{align*} 
	\end{proposition}
	
	\begin{remark}
	The statement of Proposition \ref{infinity} stays valid in a more general case where condition  \eqref{BPVE::non_ext} is replaced the the assumption of supercriticality 
	 (i.e. $ \mathbb{E} \zeta_{1,1} > 1 $) of the branching process.
	\end{remark}
	
	\begin{proof}[Proof]
		Without loss of generality, we may assume that $ \zeta_{1,1} $ takes two values only,  $ 1$ and $ 2 $  (if this is not the case, we may consider a ``minorising'' branching mechanism where $ \zeta^{\prime}_{n,i} = \min (\zeta_{n,i}, 2) $ and prove the statement for it.)
		
		For any $ C > 0 $, we let $ \gamma = \mathbb{P} \left( \xi < -C \right) $. Consider an auxiliary ``thinned'' branching process defined as follows: every edge $ e $ in the genealogical tree is deleted if  $ \xi_{e} < -C $. Then, in the thinned branching process, the number of the offspring of any vertex may take only three values:  $2$ with probability $ q(1-\gamma)^2 $, $ 1 $ with probability $ (1-q)(1-\gamma) + 2q \gamma (1-\gamma) $ and $ 0 $ with the remaining probability. Therefore, the mean number of the offspring in the thinned process,  $ E = 2q(1-\gamma)^2+2q\gamma(1-\gamma)+(1-q)(1-\gamma) =
		(1+q)(1-\gamma) $, may be made as close to $ 1+q>1 $ by taking $ C $ sufficiently large. Further, the extinction probability of the thinned process tends to  $ 0 $ as $ C \to \infty $.
		
		Let us assume that $E>1$. Then the thinned branching process $\widetilde{Z}_n$ is supercritical and  $\widetilde{Z_n}/E^n$ converges a.s. to a proper r.v., say 
		$W$, where the probability ${\mathbb P} (W=0)$ may be made as small as one wishes by taking  $ C $ sufficiently large. Therefore, for any sufficiently small $\varepsilon>0$ and $\delta>0$, one can choose $C$ and $ N $ such that
		$E>1$ and ${\mathbb P}(W>\delta) > 1-\varepsilon$, and 
		\begin{align*}
			{\mathbb P} (W>\delta, \widetilde{Z}_n\ge \delta E^n/2, \ \mbox{for all} \ n\ge N) \ge 1-2\varepsilon.
		\end{align*}
		
		Consider now the BRW that corresponds to the branching process  $\widetilde{Z}_n$. Its
		increments $\widetilde{\xi}_{n,i,j}$ form a family of i.i.d. r.v.'s with common distribution 
		\[
		{\mathbb P} (\widetilde{\xi} \le x) = {\mathbb P}(\xi \le x \ | \ \xi >-C) =
		\frac{{\mathbb P} (-C<\xi\le x)}{{\mathbb P} (\xi >C)}, \quad  \mbox{for} \ x>C.
		\]
		Consider the following family of events: for $n=1,2, \ldots$,
		\begin{align*}
			A_n = \cup_{j=1}^{\delta E^n/2} \{\widetilde{\xi}_{n,j} > 2n(C+c)\}.
		\end{align*}
		These events are mutually independent and, by the $0-1$ law, occur infinitely often a.s. since for $K=2(C+c)$ and $T=\delta/2$, 
		\begin{align*}
			\limsup_{n\to\infty}{\mathbb P} (A_n) &= 
			1 - \liminf_{n\to\infty}(1-\overline{F}(Kn)/\overline{F}(C))^{TE^n}\\
			&= 1- \exp \left( -\limsup_{n\to\infty} \overline{F}(Kn)){TE^n}/\overline{F}(C) \right)\\
			&= 1,
		\end{align*} 
		by Theorem 2.6 from \cite{FossKorZach2013}.
		
		Let $\widetilde{r}_n$ be the rightmost point in the $n$'th generation of the auxiliary BRW that corresponds to the branching process $\widetilde{Z}_n, n\ge 1$ and has increments with the common distribution function 
		${\mathbb P} (\xi \le x+c \ | \ \xi > -C)$, and let
		$\widetilde{R} = \sup_n r_n$ (where $\widetilde{r}_n=-\infty$ if $\widetilde{Z}_n=0$). 
		Then, for any  $x$ and $\varepsilon >0$,
		\begin{align*}
			{\mathbb P} (R^{(c)}>x) &\ge {\mathbb P} (\widetilde{R}>x)\\
			&\ge {\mathbb P} \left(\{ \cup_{n\ge \max (N,x/(C+c))} A_n\}\cap \{W>\delta, \widetilde{Z}_n\ge \delta E^n/2 \ \mbox{for all} \ n\ge N\} \right) \\
			&\ge 1-2\varepsilon.
		\end{align*}
		Letting $ \varepsilon $ tend to $ 0 $, we obtain the desired result.
		
	\end{proof}
	
	The authors thank Vladimir Vatutin and Bastien Mallein for useful comments, and the anonymous referee for constructive remarks and bringing our attention to references \cite{Durr1983} and \cite{Gantert2000}.

\end{document}